\documentclass[11pt,reqno]{amsart}%
\usepackage{graphicx}
\usepackage{tikz-cd}
\usepackage{amsmath, amsfonts, amssymb, amsthm}
\usepackage{enumerate, url}
\usepackage{kbordermatrix, mathdots, stmaryrd}
\usepackage[left=1in, right=1in, top=1in, bottom=0.75in]{geometry}%
\providecommand{\U}[1]{\protect\rule{.1in}{.1in}}
\newtheorem{theorem}{Theorem}[section]
\newtheorem{proposition}[theorem]{Proposition}
\newtheorem{proposition/definition}[theorem]{Proposition/Definition}
\newtheorem{lemma}[theorem]{Lemma}
\newtheorem{corollary}[theorem]{Corollary}
\newtheorem{conjecture}[theorem]{Conjecture}

\theoremstyle{definition}

\newtheorem{example}[theorem]{Example}

\theoremstyle{remark}

\makeatletter
\renewcommand*\env@matrix[1][*\c@MaxMatrixCols c]{%
  \hskip -\arraycolsep
  \let\@ifnextchar\new@ifnextchar
  \array{#1}}
\makeatother

\begin{document}
\title{Every matrix is a product of Toeplitz matrices}
\author[K.~Ye]{Ke~Ye}
\address{Department of Mathematics, University of Chicago, Chicago, IL 60637-1514.}
\email{kye@math.uchicago.edu}
\author[L.-H.~Lim]{Lek-Heng~Lim}
\address{Computational and Applied Mathematics Initiative, Department of Statistics,
University of Chicago, Chicago, IL 60637-1514.}
\email[corresponding author]{lekheng@galton.uchicago.edu}
\begin{abstract}
We show that every $n \times n$ matrix is generically a product of $\lfloor
n/2 \rfloor+ 1$ Toeplitz matrices and always a product of at most $2n+5$
Toeplitz matrices. The same result holds true if the word `Toeplitz' is
replaced by `Hankel', and the generic bound  $\lfloor
n/2 \rfloor+ 1$ is sharp. We will see that these decompositions into Toeplitz or Hankel
factors are unusual: We may not in general replace the subspace of Toeplitz or Hankel
matrices by an arbitrary $(2n-1)$-dimensional subspace of ${n \times n}$ matrices.
Furthermore such decompositions do not exist if we require the  factors to be symmetric Toeplitz, persymmetric Hankel,
or circulant matrices, even if we allow an infinite number of factors. Lastly, we discuss how the Toeplitz and Hankel decompositions of a generic
matrix may be computed by either (i) solving a system of linear and quadratic equations if the number of factors is required to
be $\lfloor n/2\rfloor + 1$, or (ii) Gaussian elimination in $O(n^3)$ time if the number of factors is allowed to be $2n$.
\end{abstract}
\maketitle

One of the top ten algorithms of the 20th century \cite{10} is the
`decompositional approach to matrix computation' \cite{Decomp}. The fact that
a matrix may be expressed as a product of a lower-triangular with an
upper-triangular matrix (LU), or of an orthogonal with an upper-triangular
matrix (QR), or of two orthogonal matrices with a diagonal one (SVD), is a
cornerstone of modern numerical computations. As aptly described in
\cite{Decomp}, such matrix decompositions provide a platform on which a variety of scientific and engineering problems can be solved. Once computed,
they may be reused repeatedly to solve new problems involving the original matrix, and
may often be updated or downdated with respect to small changes in the
original matrix. Furthermore, they permit reasonably simple rounding-error
analysis and afford high-quality software implementations.

In this article, we describe a new class of matrix decompositions that differs from the classical ones mentioned above but are similar in spirit. Recall that a Toeplitz matrix is one whose entries are constant along the diagonals and a Hankel matrix is one whose entries are constant along the reverse diagonals:
\[
T=
\begin{bmatrix}
t_{0} & t_{1} &  & t_{n}\\
t_{-1} & t_{0} & \ddots & \\
& \ddots & \ddots & t_{1}\\
t_{-n} &  & t_{-1} & t_{0}
\end{bmatrix}
,\qquad H=
\begin{bmatrix}
h_{n} &  & h_{1} & h_{0}\\
& \iddots & h_{0} & h_{-1}\\
h_{1} & \iddots & \iddots & \\
h_{0} & h_{-1} &  & h_{-n}
\end{bmatrix}
\]
Given any $n\times n$ matrix $A$ over $\mathbb{C}$, we will
show that
\begin{equation}
A=T_{1}T_{2}\cdots T_{r}, \label{eq:toep}
\end{equation}
where $T_{1},\dots,T_{r}$ are all Toeplitz matrices and
\begin{equation}
A=H_{1}H_{2}\cdots H_{r}, \label{eq:hank}
\end{equation}
where $H_{1},\dots,H_{r}$ are all Hankel matrices. We shall call
\eqref{eq:toep} a \textit{Toeplitz decomposition} and \eqref{eq:hank} a
\textit{Hankel decomposition} of $A$ accordingly. The number $r$ is $\lfloor
n/2\rfloor+1$ for almost all $n\times n$ matrices (in fact holds generically)
and is at most $4\lfloor n/2\rfloor+5\leq2n+5$ for all $n\times n$ matrices.
Furthermore, the generic bound $\lfloor n/2\rfloor+1$ is sharp. So every matrix
can be approximated to arbitrary accuracy by a product of $\lfloor n/2\rfloor+1$ Toeplitz matrices
and one cannot do better than $\lfloor n/2\rfloor+1$.

The perpetual value of matrix decompositions alluded to in the first paragraph
deserves some elaboration. A Toeplitz or a Hankel decomposition of a given
matrix $A$ may not be as easy to compute as LU or QR; but once computed, these
decompositions can be reused ad infinitum for any problem involving $A$. If
$A$ has a known Toeplitz decomposition with $r$ factors, one can solve linear
systems in $A$ within $O(r  n\log^2 n)$ time via any of the superfast algorithms in
\cite{fast0, fast1, fast2, fast7a, fast3, fast4, fast5, fast6, fast7b}.
The utility of specialized algorithms tied to a specific matrix should not be
underestimated; for example, two other algorithms that made the top ten list, Fast
Fourier Transform \cite{fft} and Fast Multipole Method \cite{fmm}, may be viewed
as algorithms designed specially for the Fourier transform matrix and
the Helmholtz matrix respectively.

\section{Why Toeplitz}\label{sec:why}

The choice of Toeplitz factors is natural for two reasons. Firstly,  Toeplitz matrices are ubiquitous and one of the most well-studied and understood
classes of structured matrices. They arise in pure mathematics: algebra \cite{alg}, algebraic geometry \cite{alggeom}, analysis \cite{anlys}, combinatorics \cite{comb}, differential geometry \cite{diffgeom}, graph theory \cite{graph}, integral equations \cite{inteqn}, operator algebra \cite{opalg}, partial differential equations \cite{pde}, probability \cite{prob}, representation theory \cite{repth}, topology \cite{top}; as well as in applied mathematics: approximation theory \cite{approx}, compressive sensing \cite{compress}, numerical integral equations \cite{numie}, numerical integration \cite{numint}, numerical partial differential equations \cite{numpde}, image processing \cite{image}, optimal control \cite{control}, quantum mechanics \cite{quantum}, queueing networks \cite{queue},  signal
processing \cite{signal}, statistics \cite{stat}, time  series analysis \cite{time}, among other areas. 

Furthermore, studies of related objects such as Toeplitz determinants \cite{toepdet2}, Toeplitz kernels \cite{toepker}, $q$-deformed Toeplitz matrices \cite{qdeform}, and Toeplitz operators \cite{toepop}, have led to much recent  success and were behind some major developments in  mathematics (e.g.\ Borodin--Okounkov formula for Toeplitz determinant \cite{toepdet1}) and in physics (e.g.\ Toeplitz quantization \cite{quantize}).

Secondly, Toeplitz matrices have some of the most attractive computational properties and are amenable to a wide range of disparate algorithms.
Multiplication, inversion, determinant computation, LU and QR decompositions of $n\times n$ Toeplitz matrices may all be computed in $O(n^{2})$ time and in
numerically stable ways. Contrast this with the usual $O(n^{3})$ complexity
for arbitrary matrices. In an astounding article  \cite{fast0}, Bitmead and Anderson first showed that Toeplitz systems may in fact be solved in $O(n \log^{2} n)$ via the use of displacement rank; later advances have achieved essentially the same complexity (possibly $O(n \log^{3} n)$) but are practically more efficient. These algorithms are based on a variety of different techniques:  Bareiss algorithm \cite{fast1}, generalized Schur algorithm \cite{fast2}, FFT and Hadamard product \cite{fast3}, Schur complement \cite{fast4}, semiseparable matrices \cite{fast5}, divide-and-concur technique \cite{fast6} --- the last three have the added advantage of mathematically proven numerical stability. One can also find algorithms based on more unusual techniques, e.g., number theoretic transforms \cite{fast7a} or syzygy reduction \cite{fast7b}. 

In parallel to these direct methods, we should also mention the equally substantial body of work in iterative methods for Toeplitz matrices (cf. \cite{iter0, iter1, iter2} and the references therein). These are based in part on an elegant theory of optimal circulant preconditioners \cite{cond0, cond1, cond2}, which are the most complete and well-understood class of preconditioners in iterative matrix computations. In short, there is a rich plethora of highly efficient algorithms for Toeplitz matrices and the Toeplitz decomposition in \eqref{eq:toep} would often (but not always) allow one to take advantage of these algorithms for general matrices. To a large extent, what we said in this section about Toeplitz matrices also holds true for Hankel matrices.

\section{Algebraic geometry}

The classical matrix decompositions\footnote{We restrict our attention to matrix decompositions that exist
for arbitrary matrices (without requiring symmetry or positive definiteness). Of the six decompositions described in \cite{Decomp}, we discounted the Cholesky, Schur, and spectral decompositions as they do not meet these criteria.} LU, QR, and SVD correspond to the Bruhat,
Iwasawa, and Cartan decompositions of Lie groups \cite{Knapp, Howe}. In this
sense, LU, QR, SVD already exhaust the standard decompositions of Lie groups
and to go beyond these, we will have to look beyond Lie theory. The Toeplitz
and Hankel decompositions described in this article represent a new
class of matrix decompositions that do not arise rom Lie theoretic considerations
but from algebraic geometric ones.

As such, the results in this article will rely on some  very basic algebraic
geometry. Since we are writing for an applied and computational mathematics
readership, we will not assume any familiarity with algebraic geometry and
will introduce some basic terminologies in this section. Readers interested
in the details may refer to \cite{Taylor} for further information.  We will assume that we are working over $\mathbb{C}$.

Let $\mathbb{C}[x_{1},\dots, x_{n}]$ denote the ring of polynomials in
$x_{1},\dots, x_{n}$ with coefficients in $\mathbb{C}$. For $f_{1},\dots,
f_{r}\in\mathbb{C}[x_{1},\dots, x_{n}]$, the set
\[
X:=\{(a_{1},\dots, a_{n})\in\mathbb{C}^{n} : f_{j}(a_{1},\dots, a_{n})=0,\,
j=1,2,\dots, r\}
\]
is called an \textit{algebraic set} in $\mathbb{C}^{n}$ defined by
$f_{1},\dots, f_{r}$. If $I$ is the ideal generated by $f_{1},\dots, f_{r}$,
we also say that $X$ is an algebraic set defined by $I$.

It is easy to see that the collection of all algebraic sets in $\mathbb{C}^n$ is closed under arbitrary intersection 
and finite union, and contains both $\varnothing$ and $\mathbb{C}^n$. In other words, the
algebraic sets form the closed sets of a topology on $\mathbb{C}^n$ that we will 
call the \textit{Zariski topology}. It is a topology that is much coaser than the usual Euclidean 
or norm topology on $\mathbb{C}^n$. All topological notions appearing in this article, unless
otherwise specified, will be with respect to the Zariski topology.

For an algebraic set $X$ in $\mathbb{C}^{n}$ defined by an ideal $I$,  the \textit{coordinate ring} $\mathbb{C}[X]$ of $X$ is the quotient
ring $\mathbb{C}[x_{1}\dots,x_{n}]/I$; the \textit{dimension}
of $X$, denoted $\dim(X)$, is the dimension of $\mathbb{C}[X]$ as a ring.
Note that $\dim(\mathbb{C}^{n})=n$, agreeing with the usual notion of dimension. A single
point set has dimension zero.

A subset $Z$ of an algebraic set $X$ which is itself also an algebraic
set is called a \textit{closed subset} of $X$; it is called a \textit{proper closed subset} if $Z\subsetneq X$. An algebraic set is said to be \textit{irreducible} if 
it is not the union of two proper closed subsets. In this paper, an \textit{algebraic variety} will mean an irreducible algebraic set
and a \textit{subvariety} will mean an irreducible closed subset of some algebraic set.

Let $X$ and $Y$ be algebraic varieties and $\varphi:\mathbb{C}[Y]\rightarrow
\mathbb{C}[X]$ be a homomorphism of $\mathbb{C}$-algebras. Then we have an
\textit{induced map} $f:X\rightarrow Y$ defined by
\[
f(a_{1},\dots,a_{n})=\bigl(\varphi(y_{1})(a_{1} ,\dots,a_{n}),\dots
,\varphi(y_{m})(a_{1},\dots,a_{n})\bigr).
\]
In general, a map $f:X\rightarrow Y$ between two algebraic varieties $X$ and
$Y$ is said to be a \textit{morphism} if $f$ is induced by a homomorphism of
rings $\varphi:\mathbb{C}[Y]\rightarrow\mathbb{C}[X]$. Let $f$ be a morphism
between $X$ and $Y$. If its image is Zariski dense, i.e., $\overline{f(X)}=Y$, then $f$ is called a
\textit{dominant} morphism. If $f$ is bijective and $f^{-1}$ is also a
morphism, then we say that $X$ and $Y$ are \textit{isomorphic}, denoted $X
\simeq Y$, and $f$ is called an \textit{isomorphism}.

An \textit{algebraic group} is a group that is also an algebraic variety where
the multiplication and inversion operations are morphisms.

For algebraic varieties, we have the following analogue of the open mapping theorem
in complex analysis with dominant morphisms playing the role of open maps
\cite{Taylor}.

\begin{theorem}\label{open mapping theorem}
Let $f:X\rightarrow Y$ be a morphism of algebraic varieties. If
$f$ is dominant, then $f(X)$ contains an open dense subset of $Y$.
\end{theorem}

A property $P$ is said to be
\textit{generic} in an algebraic variety $X$ if the points in $X$ that do not have
property $P$ are contained in a proper subvariety $Z$ of $X$. When we use the
term generic without specifying $X$, it just means that $X = \mathbb{C}^{n}$. Formally, let $Z\subset X$ be the subset consisting of points that do not satisfy $P$. If $\overline{Z}$ is a proper closed subset of $X$, then we say that a point $x\in X-Z$ is a generic point with respect to the property $P$, or just `$x\in X-Z$ is a generic point,' if the property being discussed is understood in context. The following is an elementary but useful fact regarding generic points.
\begin{lemma}\label{lem:generic}
Let $f:X\to Y$ be a morphism of algebraic varieties where $\dim(X) \ge \dim(Y)$. If there is a point $x\in X$ such that $df_x$, the differential at $x$, has the maximal rank $\dim(Y)$, then $d f_{x'}$ will also have the maximal rank $\dim(Y)$ for any generic point $x'\in X$. 
\end{lemma}
\begin{proof}
It is obvious that the rank of $d f_x$ is of the full rank if and only if the Jacobian determinant of $f$ is nonzero at the point $x$. Since the Jacobian determinant of $f$ at $x$ is a polynomial, this means for a generic point $x'\in X$, $d f_{x'}$ is also of the full rank $\dim( Y)$.
\end{proof}

All notions in this section apply verbatim to the space of $n\times n$ matrices $\mathbb{C}^{n \times n}$
by simply regarding it as $\mathbb{C}^{n^{2}}$, or, to be pedantic,
$\mathbb{C}^{n \times n} \simeq\mathbb{C}^{n^{2}}$. Note that
$\mathbb{C}^{n \times n}$ is an algebraic variety of
dimension $n^{2}$ and matrix multiplication $\mathbb{C}^{n \times n}
\times\mathbb{C}^{n \times n} \to\mathbb{C}^{n \times n}$, $(A,B) \mapsto AB$,
is a morphism of algebraic varieties $\mathbb{C}^{2n^{2}}$ and
$\mathbb{C}^{n^{2}}$.

\section{Toeplitz decomposition of generic matrices}\label{sec:toep}

Let $\operatorname{Toep}_{n}(\mathbb{C})$ be the set of all $n\times n$
Toeplitz matrices with entries in $\mathbb{C}$, i.e., the subset of
$A=(a_{i,j})_{i,j=1}^{n}\in\mathbb{C}^{n\times n}$ defined by equations
\[
a_{i,i+r}=a_{j,j+r},
\]
where $-n+1\leq r\leq n-1$ and $1\leq r+i,r+j,i,j\leq n$. Note that
$\operatorname{Toep}_{n}(\mathbb{C})\simeq\mathbb{C}^{2n-1}$ and that
$\operatorname{Toep}_{n}(\mathbb{C})$ is a subvariety of $\mathbb{C}^{n\times
n}$. In fact, it is a \textit{linear algebraic variety} defined by linear polynomials.

$\operatorname{Toep}_{n}(\mathbb{C})$, being a linear subspace of
$\mathbb{C}^{n\times n}$, has a natural basis  $B_{k}:=(\delta_{i,j+k})_{i,j=1}^{n}$, 
$k= -n+1,-n+2,\dots,n-1$, i.e., 
\[
B_{1}=\left[
\begin{smallmatrix}
0 & 1 & 0 &  & \\
& 0 & 1 & \smash{\ddots} & \\
&  & 0 & \smash{\ddots}  & 0\\
&  &  & \smash{\ddots} & 1\\
&  &  &  & 0
\end{smallmatrix} \right],
B_{2}=\left[
\begin{smallmatrix}
0 & 0 & 1 &  & \\
& 0 & 0 & \smash{\ddots} & \\
&  & 0 & \smash{\ddots}  & 1\\
&  &  & \smash{\ddots} & 0\\
&  &  &  & 0
\end{smallmatrix} \right],
\dots,B_{n-1}=\left[
\begin{smallmatrix}
0 & 0 & 0 &  &1 \\
& 0 & 0 & \smash{\ddots} & \\
&  & 0 & \smash{\ddots}  & 0\\
&  &  & \smash{\ddots} & 0\\
&  &  &  & 0
\end{smallmatrix}\right].
\]
Note that $B_{-k}=B_{k}^{\mathsf{T}}$ and $B_{0}=I$. A Toeplitz matrix $T$ may thus be expressed as
\[
T = \sum\nolimits_{j=-n+1}^{n-1} t_{j}B_{j}.
\]

Let $A=(a_{s,t})\in\mathbb{C}^{n\times n}$ be arbitrary. Suppose $j$ is a
positive integer such that $j\leq n-1$. Then it is easy to see the effect of
left- and right-multiplications of $A$ by $B_{j}$:
\[
B_{j}A =
\left[\begin{smallmatrix}
a_{j+1,1} & a_{j+1,2} & \cdots & a_{j+1,n}\\
a_{j+2,1} & a_{j+2,2} & \cdots & a_{j+2,n}\\[-0.75ex]
\vdots & \vdots  & \ddots & \vdots \\
a_{n,1} & a_{n,2} & \cdots & a_{n,n}\\
0 & 0 & \cdots & 0\\[-0.75ex]
\vdots & \vdots  & \ddots & \vdots \\[0.25ex]
0 & 0 & \cdots & 0
\end{smallmatrix} \right],\\
AB_{j}=
\left[\begin{smallmatrix}
0 & \cdots & 0 & a_{1,1} & \cdots & a_{1,n-j}\\
0 & \cdots & 0 & a_{2,1} & \cdots & a_{2,n-j}\\
0 & \cdots & 0 & a_{3,1} & \cdots & a_{3,n-j}\\
0 & \cdots & 0 & a_{4,1} & \cdots & a_{4,n-j}\\
&\ddots & \vdots & \vdots & & \vdots\\
0 & \cdots & 0 & a_{n,1} & \cdots & a_{n,n-j}
\end{smallmatrix}\right]
.
\]
Multiplying by $B_{j}$ on the left has the effect of shifting a matrix up
(if $j$ is positive) or down (if $j$ is negative) by $\lvert j\rvert$ rows,
whereas multiplying by $B_{j}$ on the right has the effect of shifting a matrix to right (if
$j$ is positive) or to left (if $j$ is negative) by $\lvert j\rvert$ columns.

We will denote $r$-tuples of $n \times n$ Toeplitz matrices by
\[
\operatorname{Toep}_n^r(\mathbb{C}) = \underset{r\text{ copies}}{\underbrace{\operatorname{Toep}_{n}(\mathbb{C})\times\dots\times
\operatorname{Toep}_{n}(\mathbb{C})}}.
\]
This is an algebraic variety in $\mathbb{C}^{rn^2}$ (endowed 
with the Zariski topology) under the subspace topology.

\begin{theorem}
\label{thm:main} Let $\rho_r:\operatorname{Toep}_n^r(\mathbb{C})\rightarrow\mathbb{C}^{n\times n}$ be the map defined by 
$\rho_r(T_{n-r},\dots,T_{n-1})=T_{n-r}\cdots T_{n-1}$. When $r\ge\lfloor
n/2\rfloor+1$, for a generic point $\tau =(T_{n-r},\dots,T_{n-1}) \in
\operatorname{Toep}_n^r(\mathbb{C})$, the differential of $\rho_r$ at
$\tau$ is of full rank $n^{2}$. Therefore for a generic $A \in \mathbb{C}^{n \times n}$, there exists
$r=\lfloor n/2\rfloor+1$ Toeplitz matrices $T_{1},\dots,T_{r}$ such that
$A=T_{1}\cdots T_{r}$.
\end{theorem}
To prove this theorem, we first fix some notations. Let $r=\lfloor
n/2\rfloor+1$ and
denote the Toeplitz matrix occuring in the $i$th argument of $\rho_r$ by
\[
X_{n-i} := \sum\nolimits_{j=-n+1}^{n-1}x_{n-i,j}B_{j}, \quad i =1,\dots,r.
\]
The differential of $\rho_r$
at a point $\tau =(T_{n-r},\dots,T_{n-1})\in\operatorname{Toep}_n^r(\mathbb{C})$
is the linear map $(d\rho_r)_{\tau}:\operatorname{Toep}_n^r(\mathbb{C})\rightarrow\mathbb{C}^{n\times n}$,
\[
(d\rho_r)_{\tau}(X_{n-r},\dots,X_{n-1})=
\sum\nolimits_{i=1}^{r}T_{n-r}\cdots
T_{n-i-1}X_{n-i}T_{n-i+1}\cdots T_{n-1},
\]
where $X_{n-i} \in \operatorname{Toep}_{n}(\mathbb{C})$, $i =1,\dots, r$. For any given $\tau$, observe that
$(d\rho_r)_{\tau}(X_{n-r},\dots,X_{n-1})$ is an $n\times n$ matrix with entries that are linear forms in the $x_{n-i,j}$'s.
Let $L_{p,q}$ be the linear form in the $(p,q)$-th entry of this matrix. 
The statement of the theorem says that we can find a point $\tau \in\operatorname{Toep}_n^r(\mathbb{C})$ so that these linear forms are linearly independent.
For any given $\tau$, since $(d\rho_r)_{\tau}$ is a linear map from the $r(2n-1)$-dimensional $\operatorname{Toep}_n^r(\mathbb{C})$
to the $n^2$-dimensional $\mathbb{C}^{n \times n}$, we may also regard it  as an $n^{2}\times(2n-1)r$
matrix $M$.  Hence our goal is to find a point $\tau$ so that this rectangular matrix $M$ has
full rank $n^{2}$, or equivalently, $M$ has a nonzero $n^{2}\times
n^{2}$ minor.

The idea of the proof of Theorem~\ref{thm:main} is that we explicitly find such a point $\tau =(T_{n-r},\dots, T_{n-1})$, where the differential $(d\rho_r)_{\tau}$ of $\rho_r$ at $\tau$ is surjective. This implies that the differential of $\rho_r$ at a generic point is surjective, allowing us to conclude that $\rho_r$ is dominant. We then apply Theorem~\ref{open mapping theorem} to deduce that the image of $\rho_r$ contains an open dense subset of $\mathbb{C}^{n\times n}$.

As will be clear later, our choice of $\tau =(T_{n-r},\dots, T_{n-1})$ will take the form
\begin{equation}\label{eq:tau}
T_{n-i} :=B_{0}+t_{n-i}(B_{n-i}-B_{-n+i}), \quad  i=1,\dots,r,
\end{equation}
where $t_{n-i}$'s are indeterminates. We will start by computing
\[
Y_{n-i}:=T_{n-r}\cdots T_{n-i-1}X_{n-i}T_{n-i+1}\cdots T_{n-1}.
\]
To avoid clutter in the subsequent discussions,
we adopt the following abbreviation: When we write $x$'s we will mean terms involving $x_{n-i,j}$, $i=1,\dots,r$, $j=-n+1,\dots, n-1$,
and when we write $t$'s we will mean terms involving $t_{n-i}$, $i=1,\dots,r$. This convention will also apply to other lists of variables.

\begin{lemma}\label{linear term}
For $\tau =(T_{n-r},\dots, T_{n-1})$ as in \eqref{eq:tau},  we have
\begin{multline*}
Y_{n-i}=X_{n-i}+\bigl[ t_{n-r}(B_{n-r}-B_{-(n-r)})+\cdots
+t_{n-i-1}(B_{n-i-1}-B_{-(n-i-1)})\bigr] X_{n-i}\\
+X_{n-i} \bigl[ t_{n-i+1}(B_{n-i+1}-B_{-(n-i+1)})
+\cdots+t_{n-1}(B_{n-1}-B_{-(n-1)}) \bigr] + \Omega(t^{2}),
\end{multline*}
where $\Omega(t^{2})$ means terms of degrees at least two in $t$'s.
\end{lemma}

By our choice of $T_{n-j}$'s, $L_{p,q}$ is a linear form in $x$'s with
coefficients that are polynomials in $t$'s. Note that $L_{p,q}$ has the form:
\[
L_{p,q}=\sum\nolimits_{i=1}^{r}x_{n-i,q-p}+ \Omega(t)
\]
where $\Omega(t)$ denotes terms of degrees at least one in $t$'s.

By our choice of $T_{n-j}$'s, entries of the coefficient matrix $M$ are also polynomials in $t$'s,
which implies that any $n^{2}\times n^{2}$ minor of $M$ is a polynomial in $t$'s. Furthermore,
observe that the constant entries (i.e., entries without $t$'s) in $M$ are all $1$'s. Let us
examine the coefficient of the lowest degree term of these minors.

\begin{lemma}
For $\tau =(T_{n-r},\dots, T_{n-1})$ as in \eqref{eq:tau}, any $n^{2}\times n^{2}$ minor $P$ of $M$ is a polynomial in $t$'s of degree at least
$(n-1)^{2}$.
\end{lemma}

\begin{proof}
Let $d\leq(n-1)^{2}-1$ be a positive integer. It suffices to show that any
term of degree $d$ in $P$ is zero. To see this note that the minor $P$ is
the determinant of a submatrix obtained from choosing $n^{2}$ columns of $M$. Hence
terms of degree $d<(n-1)^{2}$ come from taking at least $2n$ $1$'s in $M$,
otherwise the degree would be larger than or equal to $(n-1)^{2}$. If
we take $2n$ $1$'s from $M$, then there exist $(p,q)$ and $(p^{\prime
},q^{\prime})$ such that $q-p=q^{\prime}-p^{\prime}$ with two of the $1$'s coming
from $L_{p,q}$ and $L_{p^{\prime},q^{\prime}}$. But terms arising this way
must be zero because of the discussion above. 
\end{proof}
To illustrate the proof, we consider the case $n=3$ and thus $r=\lfloor 3/2 \rfloor + 1 = 2$. In this case,
\[
L_{p,q}=x_{1,q-p}+x_{2,q-p}+\Omega(t), \quad p,q = 1,2,3.
\] 
The $9\times 10$ coefficient matrix $M$ takes the form
\[\arraycolsep=2.5pt
\kbordermatrix{
& x_{1,-2} & x_{2,-2} & x_{1,-1} & x_{2,-1} & x_{1,0} &x_{2,0} & x_{1,1} &x_{2,1}  & x_{1,2} &x_{2,2}\\
L_{1,1} & * & * & * & *&  1& 1& * &* &* &* \\
L_{1,2} & * & * & * & *& * & *& 1 & 1& *&*  \\
L_{1,3} & *&*&*&*&*&*&*&*&1&1\\
L_{2,1} & *&*&1&1&*&*&*&*&*&*\\
L_{2,2} & *&*&*&*&1&1&*&*&*&*\\
L_{2,3} & *&*&*&*&*&*&1&1&*&*\\
L_{3,1} & 1&1&*&*&*&*&*&*&*&*\\
L_{3,2} & *&*&1&1&*&*&*&*&*&*\\
L_{3,3} & *&*&*&*&1&1&*&*&*&*}
\]
where the rows correspond to the $L_{p,q}$'s and the columns correspond to the $x_{n-i,j}$'s. We have marked the locations of the $1$'s and used $*$ to denote entries of the form $\Omega(t)$. It is easy to see in this case that if we take a $9\times 9$ minor of $M$, the degree in $t$'s of this minor will be at least four. Indeed, it is also not hard to see that there exists a minor of degree exactly four --- this is the content of our next lemma.

Since a linear change of variables does not change the rank of a matrix, to
simplify our calculations, we will change our $x$'s to $y$'s linearly as follows:
\begin{align*}
y_{j}  &  =x_{n-r,j}+\dots+ x_{n-1,j},\\
y_{n-1,j}  &  =x_{n-r,j}+\dots+x_{n-2,j},\\
& \; \vdots\\
y_{n-(r-1),j}  &  =x_{n-r,j},
\end{align*}
for each $-(n-1)\le j \le n-1$.

\begin{lemma}\label{lem:exact}
For $\tau =(T_{n-r},\dots, T_{n-1})$ as in \eqref{eq:tau}, there exists an $n^{2}\times n^{2}$ minor $P$ of $M$ that contains a monomial term of degree exactly $(n-1)^{2}$ in $t$'s and whose cofficient is non-zero. It follows that $\operatorname{rank}(M) = n^2$ for this particular choice of $\tau$.
\end{lemma}

\begin{proof}
If we use more than $2n-1$ $1$'s from $M$ to form monomials in $P$, then we must obtain that coefficients of these monomials are zero. Therefore, the only way to obtain a nonzero coefficient for the degree-$(n-1)^{2}$
term is to take exactly $2n-1$ $1$'s and $(n-1)^{2}$ terms involving $t$'s to the first power.
We may thus ignore the $\Omega(t^{2})$ in Lemma~\ref{linear term}. We claim that
there exists a minor of $M$ such that it contains the monomial $t_{n-1}^{2n-2}t_{n-2}^{2n-2}\cdots t_{n-r+2}^{2n-2}t_{n-r+1}^{2r-3}$ if $n$ is even
and $t_{n-1}^{2n-2}t_{n-2}^{2n-2}\cdots t_{n-r+2}^{2n-2}t_{n-r+1}^{2n-2}$ if $n$ is odd.  We will prove the odd case.  The even case can be proved in the same manner.
Let $n=2k+1$. Then $r=\lfloor n/2\rfloor+1=k+1$ and $n-r+1=k+1$. Upon transforming to the
new coordinates $y$'s defined before this lemma, $L_{p,q}$ takes the form:
\begin{multline}
L_{p,q}=y_{q-p}   +(t_{q-1}y_{1-p,q-1}+\cdots)-(t_{n-q}y_{n-p,n-q}+\cdots)\\
 +[t_{n-p}(y_{q-n}-y_{n-p+1,q-n})+\cdots]
-[t_{p-1}(y_{q-1}-y_{p,q-1})+\cdots],\label{formula}
\end{multline}
where we have adopted the convention that $t_{i}:=0$ if it is
not of the form $t_{n-j}$ with $j=1,\dots,r$.
The `$\cdots$' in \eqref{formula} denotes the trailing terms that play no role
in the formation of the required minor $P$. For example, the trailing terms after
$t_{q-1}y_{1-p,q-1}$ are $t_{q-2}y_{2-p,q-2}+t_{q-3}y_{3-p,q-3}+\cdots
+t_{n-r}y_{p-q+r-n,n-r}$.  By \eqref{formula}, we have
to choose exactly one $1$ from the linear forms $L_{1,q-p+1},L_{2,q-p+2},\dots,L_{n,q-p+n}$, where $1\le q-p+j\le n$ and $j=1,\dots, n$. Now it is
obvious that there is only one way to obtain a monomial containing
$t_{n-1}^{2n-2}$, because $t_{n-1}$ only appears in $L_{j,1}$ and $L_{j,n}$
for $j= 1,\dots,n$. By the same reasoning, the monomial containing
$t_{n-1}^{2n-2}t_{n-2}^{2n-2}$ is unique. Continuing this procedure, we arrive at 
the conclusion that the monomial $t_{n-1}^{2n-2}t_{n-2}^{2n-2}\cdots t_{n-r+1}^{2n-2}$
is unique in $P$ and in particular the coefficient of this monomial is not zero. 
\end{proof}

To illustrate the proof of Lemma~\ref{lem:exact}, we work out the case $n=5$ explicitly. In this case, there are $25$ linear forms $L_{i,j}$ where $i,j = 1,\dots, 5$. The coefficients of these linear forms $L_{i,j}$ determine a matrix $M$ of size $25\times  27$. Each $25 \times 25$ minors of $M$ is a polynomial in $t$'s. Our goal is to find a nonzero minor $P$ of $M$ and we achieve this by finding a nonzero monomial in $P$. Since $r=\lfloor 5/2\rfloor + 1 =3$, the monomial we seek is $t_{4}^{8}t_{3}^{8}$. If we backtrack the way we calculate minors of $M$, we would see that to obtain a particular monomial, we need to take one coefficient from each $L_{i,j}$ in an appropriate way.

It is easy to see that $t_{4}$ appears in $L_{j,1}$ and $L_{j,5}$ for $ j=1,\dots,5$. Notice
that we have to take exactly one $1$ from each linear form in $\{L_{p,q}: -4 \le q-p \le 4\}$ and so we are only allowed to take
eight $t_{4}$'s from the ten linear forms $L_{j,1}$ and $L_{j,5}$, $j=1,\dots,5$, because the set  $\{L_{p,q}: q-p = s \}$ contains only one element when $s=-4$ or  $4$. Next, we need
to choose $t_{3}$, and $t_{3}$ appears only in $L_{1,j},L_{2,j},L_{4,j},L_{5,j},L_{j,1},L_{j,2},L_{j,4}$ as well as $L_{j,5}$. Since we have already used
$L_{j,1}$ and $L_{j,5}$ in the previous step, we are only allowed to take
$t_{3}$ from $L_{j,2}, L_{j,4}$ and $L_{1,j},L_{2,j},L_{4,j},L_{5,j}$.
But by \eqref{formula},  the $y$'s in $L_{1,j},L_{2,j},L_{4,j},L_{5,j}$ with coefficients involving $t_{3}$
have also been used in the previous step, compelling us to choose $L_{j,2}, L_{j,4}$.

Again, we have to take $1$'s from each linear form in $\{L_{p,q}: -4 \le q-p \le 4\}$. Therefore we obtain $t_{3}^{8}$. Now there are  five $L_{i,j}$'s left,
they are $L_{j,3}$ for $1\leq j\leq5$ and we have to take $1$ for each
$L_{j,3}$ since we need nine $1$'s. Thus, we obtain $t_{4}^{8}t_{3}^{8}$ in
the unique way.

The following summarizes the procedure explained above. The three tables below are intended to show how we obtain the monomial $t_{4}^{8}t_{3}^{8}$.  The $(i,j)$-th entry of the three tables indicates the term we pick from $L_{i,j}$. For example, the $(1,1)$-th entry of those tables means that we would pick $t_4$ from $L_{1,1}$ and the $(5,3)$-th entry means that we would pick $1$ from $L_{5,3}$. The `$\times$' in the $(1,2)$-th entry of the first table indicate that we have yet to pick an entry from $L_{1,2}$. In case it is not clear, we caution the reader that these tables are neither matrices nor determinants.
\begin{enumerate}
\item We pick eight $t_{4}$'s from $L_{i,1}$ and $L_{j,5}$ where $i=1,\dots,
4$, $j = 2,\dots,5$, and we pick two $1$'s from $L_{5,1}$ and $L_{1,5}$. This yields a factor of $t_4^8$.
{\footnotesize
\[
\begin{matrix}[c|ccccc]
L_{i,j} & 1 & 2 & 3 & 4 & 5 \\\hline
1 & t_{4} & \times & \times & \times & 1\\
2 & t_{4} & \times & \times & \times & t_{4}\\
3 & t_{4} & \times & \times & \times & t_{4}\\
4 & t_{4} & \times & \times & \times & t_{4}\\
5 & 1 & \times & \times & \times & t_{4}
\end{matrix}
\]}%

\item We pick eight $t_{3}$'s from $L_{i,2}$ and $L_{j,4}$ where $i=1,\dots,
4$, $j = 2,\dots,5$,  and we pick two $1$'s from $L_{5,2}$ and $L_{4,5}$. This yields the factor $t_4^8 t_3^8$.
{\footnotesize
\[
\begin{matrix}[c|ccccc]
L_{i,j} & 1 & 2 & 3 & 4 & 5 \\\hline
1 & t_{4} & t_{3} & \times & 1 & 1\\
2 & t_{4} & t_{3} & \times & t_{3} & t_{4}\\
3 & t_{4} & t_{3} & \times & t_{3} & t_{4}\\
4 & t_{4} & t_{3} & \times & t_{3} & t_{4}\\
5 & 1 & 1 & \times & t_{3} & t_{4}
\end{matrix}
\]}%

\item In order to preserve the  $t_4^8t_3^8$ factor obtained above, we pick five $1$'s from $L_{1,3}$, $L_{2,3}$, $L_{3,3}$, $L_{4,3}$, and $L_{5,3}$.
{\footnotesize
\[
\begin{matrix}[c|ccccc]
L_{i,j} & 1 & 2 & 3 & 4 & 5 \\\hline
1 & t_{4} & t_{3} & 1 & 1 & 1\\
2 & t_{4} & t_{3} & 1 & t_{3} & t_{4}\\
3 & t_{4} & t_{3} & 1 & t_{3} & t_{4}\\
4 & t_{4} & t_{3} & 1 & t_{3} & t_{4}\\
5 & 1 & 1 & 1 & t_{3} & t_{4}
\end{matrix}
\]}%
\end{enumerate}

\begin{proof}[Proof of Theorem~\ref{thm:main}:] By Lemma~\ref{lem:exact}, the linear map
$d\rho_{\tau}$ is surjective at the point $\tau = (T_{n-r},\dots,T_{n-1})$ as defined in \eqref{eq:tau}. Hence $d\rho_{\tau}$ is surjective at any generic
point $\tau$ by Lemma~\ref{lem:generic}. If  $\operatorname{im}(\rho)$ is contained in a closed subset
of $\mathbb{C}^{n\times n}$, then we obtain that the rank of $d\rho$ at a generic point has rank less than or equals to $n^2$, which is a contradiction to the fact that $d\rho$ is
surjective at a generic point. By Theorem~\ref{open mapping theorem}, we see that the image of $\rho$ contains an open dense subset of $\mathbb{C}^{n\times n}$. This completes the proof of
Theorem~\ref{thm:main}. 
\end{proof}

Let $X$ be a generic $n\times n$ matrix. Then  Theorem~\ref{thm:main} ensures the existence of
a decomposition into a product of $r=\lfloor n/2 \rfloor+1 $ Toeplitz
matrices.  Note that the decomposition of $X$ is not unique without further
conditions on the Toeplitz factors. An easy way to see this is that $(\alpha_1 T_1)(\alpha_2 T_2) \cdots 
(\alpha_r T_r) = T_1 T_2 \cdots T_r$ as long as $\alpha_1 \alpha_2 \cdots \alpha_r = 1$.
In fact, the preimage $\rho^{-1}(X)$ is the set of
$r$-tuples of Toeplitz matrices $(T_{1},T_{2},\dots, T_{r})$ such that
$T_{1}T_{2}\cdots T_{r}=X$, and this set is an algebraic set of dimension $r(2n-1)-n^{2}$, i.e., $3n/2 -1$ for even $n$ and $(n-1)/2$ for odd $n$.

The generic number of Toeplitz factor $r = \lfloor n/2 \rfloor + 1$ in Theorem~\ref{thm:main} is however sharp.
\begin{corollary}
\label{cor:sharp-t} $r=\lfloor n/2 \rfloor+1$ is the smallest integer such
that every generic $n\times n$ matrix is a product of
$r $ Toeplitz matrices.
\end{corollary}

\begin{proof}
If $r$ is not the smallest such integer, then there exists some $s<r$
such that $\rho_s:\operatorname{Toep}_{n}(\mathbb{C})^{s}\rightarrow
\mathbb{C}^{n \times n}$ is dominant, i.e., the image of $\rho_s$ is dense in $\mathbb{C}^{n \times n}$. Since $\rho_s$ is
a polynomial map with a dense image, it is a morphism between two algebraic varieties and hence its
image contains an open dense subset of $\mathbb{C}^{n \times n}$. This implies that
$\dim(\operatorname{Toep}_{n}(\mathbb{C})^{s})\geq\dim(\mathbb{C}^{n \times n})$, i.e., $s^2 \ge n^2$, contradicting our assumption that $s<r=\lfloor n/2 \rfloor$. 
\end{proof}

Although we have been working over $\mathbb{C}$ for convenience, results we obtained in this paper hold over any algebraically closed field. Indeed, Theorem~\ref{open mapping theorem} is true for any morphism of schemes over an integral domain (see, for example, \cite{Taylor}) and Lemma~\ref{lem:generic} is true over any infinite perfect field (see, for example, \cite{Geck}). In other words, the two results that we use in our proofs here hold over algebraically closed fields.

Moreover, even though the proof of Theorem~\ref{thm:main} requires algebraic closure, if we only consider the dominance and surjectivity of $\rho_r$ as a morphism of schemes, then our results are true over any infinite field of characteristic zero since Theorem~\ref{open mapping theorem} and Lemma~\ref{lem:generic} hold in this case. For example, it is true that the image of $\rho_r$ contains an open subscheme of $\mathbb{C}^{n \times n}$, but this does not imply that a generic matrix is the product of $r$ Toeplitz matrices. The reason being that for a non-algebraically field $\Bbbk$,  there is no one-to-one correspondence between closed points of $\operatorname{Spec} (\Bbbk[x_1,\dots, x_n])$ and elements of $\Bbbk^n$ (such a correspondence exists for an algebraically closed field).

\section{Hankel decomposition of generic matrices}\label{sec:hank}

A Hankel matrix $H=(h_{ij})\in
\mathbb{C}^{n\times n}$ is one that satisfies $h_{ij}=h_{i^{\prime}j^{\prime}}$ whenever
$i+j=i^{\prime}+j^{\prime}$. The set of all $n\times n$ Hankel matrices,
denoted by $\operatorname{Hank}_{n}(\mathbb{C})$, is a linear subvariety of
$\mathbb{C}^{n\times n}$ defined by equations $x_{ij}=x_{i+p,j-p}$ for all
$1\leq i,j\leq n$ and $\max\{1-i,j-n\}\leq p\leq\min\{n-i,j-1\}$. As is the
case for Toeplitz matrices, $\operatorname{Hank}_{n}(\mathbb{C})
\simeq\mathbb{C}^{2n-1}$ and is a linear subspace of $\mathbb{C}^{n \times n}$.

Here we consider the analogous decomposition problem for Hankel matrices:
Find the smallest natural number $r$ such that the map
\begin{equation}\label{eq:hdecom}
\rho_r:\operatorname{Hank}_n^r(\mathbb{C})\rightarrow\mathbb{C}^{n\times
n}
\end{equation}
sending $(H_{1},\dots,H_{r})$ to the product $H_{1}\cdots H_{r}$, is
generically surjective (we abused notation slightly by using $\rho_r$ to denote 
the product map here even though the domain is different from before). Again the superscript is intended to denote product:
\[
\operatorname{Hank}_n^r(\mathbb{C}) = \underset{r\text{ copies}}{\underbrace{\operatorname{Hank}_{n}(\mathbb{C})\times\dots\times
\operatorname{Hank}_{n}(\mathbb{C})}}.
\]
As noted above, the dimension of $\operatorname{Hank}_{n}(\mathbb{C})$
is $2n-1$, so $r$ must be at least $\lfloor n/2\rfloor+1$. We show in the following that $r=\lfloor n/2\rfloor+1$ by
reducing the problem to the Toeplitz case.

We first introduce three linear operators on $\mathbb{C}^{n\times n}$ that are analogues of the matrix
transpose. For $A = (a_{i,j}) \in\mathbb{C}^{n\times n}$, these are defined as follows.
\begin{description}
\item[Transpose] $A^{\mathsf{T}}=(a_{j,i})$.

\item[Rotate] $A^{\mathsf{R}}=(a_{n+1-j,i})$.

\item[Swap] $A^{\mathsf{S}}=(a_{i,n+1-j})$.

\item[Flip] $A^{\mathsf{F}}=(a_{n+1-i,j})$.
\end{description}

For illustration,
\[
A=
\begin{bmatrix}
1 & 2 & 3\\
4 & 5 & 6\\
7 & 8 & 9
\end{bmatrix}
,\quad 
A^{\mathsf{T}}=
\begin{bmatrix}
1 & 4 & 7\\
2 & 5 & 8\\
3 & 6 & 9
\end{bmatrix}
,\quad 
A^{\mathsf{R}}=
\begin{bmatrix}
3 & 6 & 9\\
2 & 5 & 8\\
1 & 4 & 7
\end{bmatrix}
,\quad
A^{\mathsf{S}}=
\begin{bmatrix}
3 & 2 & 1\\
6 & 5 & 4\\
9 & 8 & 7
\end{bmatrix}
,\quad
A^{\mathsf{F}}=
\begin{bmatrix}
7 & 8 & 9\\
4 & 5 & 6\\
1 & 2 & 3
\end{bmatrix}
.
\]
These define linear maps of $\mathbb{C}^{n \times n} \to \mathbb{C}^{n \times n}$ that are clearly also isomorphisms of varieties. 
Moreover we have
\[
A^{\mathsf{SS}}=A^{\mathsf{FF}}=A^{\mathsf{RRRR}} = A.
\]
Note that we write $A^{\mathsf{XY}}$ to mean $(A^{\mathsf{X}})^{\mathsf{Y}}$ for any $\mathsf{X}, \mathsf{Y} \in \{ \mathsf{F},\mathsf{R},\mathsf{S},\mathsf{T} \}$.

Observe that $H\in\operatorname{Hank}_{n}(\mathbb{C})$ if and
only if $H^{\mathsf{R}}\in\operatorname{Toep}_{n}(\mathbb{C})$, $H^{\mathsf{S}}\in\operatorname{Toep}_{n}(\mathbb{C})$, or $H^{\mathsf{F}}\in
\operatorname{Toep}_{n}(\mathbb{C})$.

\begin{lemma}
\label{identity1} Let $A$ and $B \in\mathbb{C}^{n\times n}$. Then
\begin{enumerate}[\upshape (i)]
\item $(AB)^{\mathsf{R}}=B^{\mathsf{RS}}A^{\mathsf{R}}={B^{\mathsf{R}}}A^{\mathsf{RF}}$;

\item $A^{\mathsf{SR}}=A^{\mathsf{T}}$;

\item $A^{\mathsf{FR}}=A^{\mathsf{T}}$;

\item $(AB)^{\mathsf{S}}=AB^{\mathsf{S}}$;

\item $(AB)^{\mathsf{F}}=A^{\mathsf{F}}B$.
\end{enumerate}
\end{lemma}

The lemma below allows us to deduce a corresponding theorem for Hankel matrices from Theorem~\ref{thm:main}.
\begin{lemma}\label{lemma2}
Let $A_{1},\dots,A_{r}\in\mathbb{C}^{n\times n}$ matrices. Then
\[
(A_{1}^{\mathsf{S}}\cdots A_{r}^{\mathsf{S}})^{\mathsf{R}}=A_{r}^{\mathsf{SR}} A_{r-1}^{\mathsf{SRSF}}(A_{1}^{\mathsf{S}}\cdots
A_{r-2}^{\mathsf{S}})^{\mathsf{R}}
\]

\end{lemma}

\begin{proof}
By Lemma~\ref{identity1} (i), (iv) and (v), we have
\[
(A_{1}\cdots A_{r})^{\mathsf{R}}=A_{r}^{\mathsf{SR}}
(A_{1}^{\mathsf{S}}\cdots A_{r-1}^{\mathsf{S}})^{\mathsf{RF}}
=A_{r}^{\mathsf{SR}} A_{r-1}^{\mathsf{SRSF}}(A_{1}^{\mathsf{S}}\cdots A_{r-2}^{\mathsf{S}})^{\mathsf{R}},
\]
as required. 
\end{proof}

By Lemma~\ref{lemma2} and induction on $r$, we see that a product of
Toeplitz matrices can be translated into a product of Hankel matrices via the
rotate operator, which allows us to deduce the following.
\begin{theorem}
Let $\rho_r:\operatorname{Hank}_n^r(\mathbb{C})\rightarrow
\mathbb{C}^{n\times n}$ be the map defined by $\rho_r(H_{1},\dots
,H_{r})=H_{1}\cdots H_{r}$. If $r\geq\lfloor n/2\rfloor+1$, then $\rho_r$ is
dominant. Therefore, for a generic $A \in \mathbb{C}^{n \times n}$, there exists $r=\lfloor
n/2\rfloor+1$ Hankel matrices $H_{1},\dots,H_{r}$ such that $A=H_{1}\cdots
H_{r}$.
\end{theorem}

\begin{proof}
Consider the map
\[
f:\operatorname{Hank}_n^r(\mathbb{C})\xrightarrow{\mathsf{S}}\operatorname{Toep}_n^r(\mathbb{C})
\xrightarrow{\rho_r}\mathbb{C}^{n\times n}\xrightarrow{\mathsf{R}}\mathbb{C}^{n\times n}.
\]
Here $\mathsf{S}$ denotes the swap operator and $\mathsf{R}$ denotes the
rotate operator. By Lemma~\ref{lemma2} and induction on $r$, we see that
$\operatorname{im}(f)\simeq\rho_r(\operatorname{Toep}_n^r(\mathbb{C}))$ which
is in turn isomorphic to $\rho_r(\operatorname{Hank}_n^r(\mathbb{C}))$. Hence
$\rho_r(\operatorname{Hank}_n^r(\mathbb{C}))$ is dense in $\mathbb{C}^{n\times n}$ by Theorem~\ref{thm:main}. 
\end{proof}

\begin{corollary}
\label{cor:sharp-h} $r=\lfloor n/2\rfloor+1$ is the smallest integer such that
every generic $n\times n$ matrix is a product of $r$ Hankel matrices.
\end{corollary}
\begin{proof} Same as that for Toeplitz matrices.  \end{proof}

\section{Toeplitz and Hankel decompositions of arbitrary matrices}

We now show that every invertible $n\times n$ matrix is a
product of $2r$ Toeplitz matrices and every matrix is a product
of $4r+1$ Toeplitz matrices, where $r=\lfloor n/2\rfloor+1$. The same results
hold also for Hankel matrices.

We make use of the following property of algebraic groups \cite{Borel}.
\begin{lemma}
\label{the product of twe open dense subsets is the group} Let $G$ be an
algebraic group and $U,V$ be two open dense subsets of $G$. Then $UV=G$.
\end{lemma}

\begin{proposition}
\label{prop:decomp} Let $W$ be a subspace of $\mathbb{C}^{n\times
n}$ such that the map $\rho:W^{r}\rightarrow\mathbb{C}^{n\times n}$ is
dominant. Then every invertible $n\times n$ matrix can be expressed as the
product of $2m$ elements in $W$.
\end{proposition}
\begin{proof}
Since $\rho$ is dominant, $\operatorname{im}(\rho)$ contains an open dense subset
of $\mathbb{C}^{n\times n}$. On the other hand, $\operatorname{GL}_{n}(\mathbb{C})$ is an
open dense subset of $\mathbb{C}^{n\times n}$, therefore $\operatorname{im}
(\rho)$ contains an open dense subset of $\operatorname{GL}_{n}(\mathbb{C})$. Let $U$ be
such an open dense subset. Then by Lemma~\ref{the product of twe open dense subsets is the group}
we see that
$UU=\operatorname{GL}_{n}(\mathbb{C})$. Hence every invertible matrix $A$ can be expressed
as a product of two matrices in $U$ and so $A$ can be expressed as a product
of $2m$ matrices in $W$. 
\end{proof}

\begin{corollary}
\label{decomposition of invertible matrices} Every invertible $n\times n$
matrix can be expressed as a product of $2r$ Toeplitz matrices or $2r$ Hankel matrices for $r=\lfloor
n/2 \rfloor+1$.
\end{corollary}

\begin{proof}
By Theorem~\ref{thm:main}, we have seen that the map $\rho$ is dominant. Hence
by Proposition~\ref{prop:decomp}, every invertible matrix is a
product of $2r$ Toeplitz matrices. Similarly for Hankel matrices. 
\end{proof}

\begin{lemma}
\label{decomposition of arbitrary matrices} Let $W$ be a linear subspace of
$\mathbb{C}^{n\times n}$ such that $\rho:W^{r}\rightarrow\mathbb{C}^{n\times
n}$ is dominant. Let $A \in \mathbb{C}^{n\times n}$ and suppose the orbit
of $A$ under the action of $\operatorname{GL}_{n}(\mathbb{C})\times\operatorname{GL}_{n}(\mathbb{C})$,
acting by left and right matrix multiplication, intersects $W$. Then $A$ can
be expressed as a product of $4m+1$ matrices in $W$.
\end{lemma}
\begin{proof}
By assumption, there exist invertible matrices $P,Q$ such that $A=PBQ$ where
$B\in W$. By Proposition~\ref{decomposition of invertible matrices}, we know
that $P,Q$ can be decomposed into a product of $r$ matrices in $W$.
Hence $A$ can be expressed as a product of $4m+1$ matrices in
$W$. 
\end{proof}

\begin{theorem}
\label{thm:every-t} Every $n\times n$ matrix can be expressed as a product of
$4r+1$ Toeplitz matrices or $4r+1$ Hankel matrices for $r=\lfloor n/2 \rfloor+1$.
\end{theorem}

\begin{proof}
It remains to consider the rank deficient case.
Let $A$ be an $n\times n$ matrix of rank $m < n$, then there exist invertible
matrices $P,Q$ such that $A=P B_{n-m} Q$, where $B_{k}=(\delta_{i+k,j})$ for
$k=1,\dots, n-1$. By Lemma~\ref{decomposition of arbitrary matrices}, $A$ is
a product of $4r+1$ Toeplitz matrices. Note that the Hankel matrix $B_{k}^{\mathsf{S}}$ can be written as $B_{k} \Pi$ for some permutation
matrix $\Pi$ and the same argument can be applied to $A = P B_{n-m}^{\mathsf{S}} (\Pi^{\mathsf{T}} Q)$.  
\end{proof}

It is easy to see that $4r+1$ is not the smallest integer $p$ such that every
$n\times n$ matrix is a product of $p$ Toeplitz matrices. For example,
consider the case $n=2$.  If we set
\[
\begin{bmatrix}
x & y\\
z & x
\end{bmatrix}
\begin{bmatrix}
s & t\\
u & s
\end{bmatrix}
=
\begin{bmatrix}
a & b\\
c & d
\end{bmatrix}
\]
where $x,y,z,s,t,u$ are unknowns, a simple calculation shows that when $c=b=0$
we have a solution
\[
\begin{bmatrix}
0 & a\\
d & 0
\end{bmatrix}
\begin{bmatrix}
0 & 1\\
1 & 0
\end{bmatrix}
=
\begin{bmatrix}
a & 0\\
0 & d
\end{bmatrix}
,
\]
and otherwise we have solutions
\[
x=\frac{as-bu}{s^{2}-tu},\quad y=\frac{bs-at}{s^{2}-tu},\quad z=\frac
{cs^{2}-ctu-asu+bu^{2}}{(s^{2}-tu)s},
\]
where $s,t,u$ are parameters satisfying
\[
  (s^{2}-tu)s\neq0,\quad
  (a-d)s^{3}+cs^{2}t-bs^{2}u-ct^{2}u+btu^{2}+(d-a)stu=0.
\]
Hence any $2\times2$ matrix requires two Toeplitz factors to decompose. 

While the generic bound $r = \lfloor n/2 \rfloor+ 1$ is sharp by
Corollaries~\ref{cor:sharp-t} and \ref{cor:sharp-h}, we see no reason that the bound $4r+1$ in Theorem~\ref{thm:every-t} should also be sharp. In fact, we are optimistic that the generic bound $r$ holds \textit{always}:
\begin{conjecture}\label{conj}
Every matrix $A \in\mathbb{C}^{n \times n}$ is a product of at most $\lfloor
n/2 \rfloor+ 1$ Toeplitz matrices or $\lfloor
n/2 \rfloor+ 1$ Hankel matrices.
\end{conjecture}

\section{Toeplitz and Hankel decompositions are special}

We will see in this section that the Toeplitz and Hankel decompositions studied above are exceptional in two ways: (i) The Toeplitz or Hankel structure of the factors cannot be extended to arbitrary structured matrices that form a $(2n-1)$-dimensional subspace of $\mathbb{C}^{n \times n}$; (ii) The Toeplitz or Hankel structure of the factors cannot be further restricted to circulant, symmetric Toeplitz, or persymmetric Hankel --- even though these all seem plausible at first. Moreover (i) and (ii) hold even if we allow an infinite number of factors in the decomposition.

Noting that  $\operatorname{Toep}_{n}(\mathbb{C})$  and $\operatorname{Hank}_{n}(\mathbb{C})$ are both $(2n-1)$-dimensional subspaces of $\mathbb{C}^{n \times n}$, one might suspect that such decompositions are nothing special and would hold for any subspace $W \subseteq \mathbb{C}^{n \times n}$ of dimension $2n-1$. This is not the case. In fact, for any $d = 1,\dots, n^2-n+1$, we may easily construct a $d$-dimensional subspace $W \subseteq \mathbb{C}^{n \times n}$  whereby a decomposition of an arbitrary matrix into  products of $r$ elements of $W$ does not exist for any $r \in \mathbb{N}$. For example, $W$ could be taken to be any $d$-dimensional subspace of $\mathbb{C}^{n \times n}$ consisting of matrices of the form
\[
\begin{bmatrix}
*& *& \cdots &*\\
0& *&\cdots &*\\
\vdots & \vdots &\ddots&\vdots\\
0 & * &\cdots &*
\end{bmatrix},
\]
i.e., with zeros below the $(1,1)$-entry. Since such a structure is preserved under matrix product, the semigroup generated by $W$, i.e., the set of all products of matrices from $W$, could never be equal to all of $\mathbb{C}^{n \times n}$.

While here we are primarily concern with the semigroup generated by a subspace, it is interesting to also observe the following.
\begin{proposition}
Let $W$ be a proper associative subalgebra (with identity) of $\mathbb{C}^{n \times n}$.  Then $\dim W\le n^2-n+1$.
\end{proposition}
\begin{proof}
Every associative algebra can be made into a Lie algebra by defining the Lie bracket as $[X,Y]=XY-YX$. So $W$ may be taken to be a Lie algebra. Let $\mathfrak{sl}_n(\mathbb{C})$ be the Lie algebra of traceless matrices. For any $X\in W$, we can write
\[
X=X_0+\frac{\operatorname{tr}(X)}{n} I,
\]
where $\operatorname{tr}(X)$ is the trace of $X$, $X_0$ is an element in $\mathfrak{sl}_n(\mathbb{C})$, and $I$ is the identity matrix. In particular, $X_0 \in W$ since both $I$ and  $X$ are in $W$. Hence we have 
\[
W= (W\cap \mathfrak{sl}_n (\mathbb{C}))\oplus \mathbb{C}\cdot I.
\]
Since $W\cap \mathfrak{sl}_n(\mathbb{C})$ is a proper Lie subalgebra of $\mathfrak{sl}_n(\mathbb{C})$ and the dimension of a proper Lie subalgebra of $\mathfrak{sl}_n(\mathbb{C})$ cannot exceed $n^2-n$ \cite{BS}, we must have
\[
\dim W\le n^2-n+1.\qedhere
\]
\end{proof}

On the other hand, one might perhaps think that any $n\times n$ matrix is expressible as a product of $n$ \textit{symmetric} Toeplitz matrices (note that these require exactly $n$ parameters to specify and form an $n$-dimensional linear subspace of $\operatorname{Toep}_n(\mathbb{C})$. We see below that this is false.

\begin{theorem}
Let $n \ge 2$. There exists $A \in \mathbb{C}^{n \times n}$ that cannot be expressed as a product of $r$ symmetric Toeplitz matrices for any $r \in \mathbb{N}$.
\end{theorem}
\begin{proof}
We exhibit a subset $S \subsetneq \mathbb{C}^{n \times n}$  that contains all symmetric Toeplitz matrices but also matrices that are neither symmetric nor Toeplitz. The desired result then follows by observing that there are $n \times n$ matrices that cannot be expressed as a product of elements from $S$.

Let the entries of $X, Y \in \mathbb{C}^{n \times n}$ satisfy $x_{ij}=x_{n-i+1,n-j+1}$ and $y_{ij}=y_{n-i+1,n-j+1}$ respectively, i.e.,
{\footnotesize
\[
X=\begin{bmatrix}
x_{11}    & x_{12}     & x_{13}  & \cdots &x_{1,n-2} & x_{1,n-1} & x_{1n}\\
x_{21}& x_{22}     & x_{23}  & \cdots &x_{2,n-2} & x_{2,n-1} & x_{2n}\\
x_{31}& x_{32} & x_{33}& \cdots &x_{3,n-2} &x_{3,n-1}  &x_{3n}\\
\vdots& \vdots & \vdots& \ddots &\vdots &\vdots &\vdots\\
x_{3n}& x_{3,n-1}& x_{3,n-2} & \cdots &x_{33} & x_{32} & x_{31}\\
x_{2n}& x_{2,n-1}& x_{2,n-2} & \cdots &x_{23} &x_{22} & x_{21}\\
x_{1n}& x_{1,n-1} & x_{1,n-2 }& \cdots&x_{13} & x_{12} & x_{11}  
\end{bmatrix},
Y=\begin{bmatrix}
y_{11}    & y_{12}     & y_{13}  & \cdots &y_{1,n-2} & y_{1,n-1} & y_{1n}\\
y_{21}& y_{22}     & y_{23}  & \cdots &y_{2,n-2} & y_{2,n-1} & y_{2n}\\
y_{31}& y_{32} & y_{33}& \cdots &y_{3,n-2} &y_{3,n-1}  &y_{3n}\\
\vdots& \vdots & \vdots& \ddots &\vdots &\vdots &\vdots\\
y_{3n}& y_{3,n-1}& y_{3,n-2} & \cdots &y_{33} & y_{32} & y_{31}\\
y_{2n}& y_{2,n-1}& y_{2,n-2} & \cdots &y_{23} &y_{22} & y_{21}\\
y_{1n}& y_{1,n-1} & y_{1,n-2 }& \cdots&y_{13} & y_{12} & y_{11}  
\end{bmatrix}.
\]}%
Let $Z=(z_{ij}) = XY$. Then it is easy to see that $z_{ij}=z_{n-i+1,n-j+1}$ since 
\[
z_{ij}=\sum_{k=1}^n x_{ik} y_{kj}=\sum_{k=1}^n x_{n-i+1,n-k+1} y_{n-k+1,n-j+1}= z_{n-i+1, n-j+1}.
\]
Let $S$ be the variety of matrices defined by equations $x_{ij}=x_{n-i+1,n-j+1}$ where $1 \le i,j \le n$. It is obvious that $S$ is a proper subvariety of $\mathbb{C}^{n\times n}$ and we just saw that it is closed under matrix product, i.e., $X,Y\in S$ implies $XY\in S$.

Since symmetric Toeplitz matrices are contained in $S$, product of any $r$ symmetric Toeplitz matrices must also be in $S$. Therefore for any $r\in \mathbb{N}$ and $A\not\in S$, it is not possible to express $A$ as a product of $r$ symmetric Toeplitz matrices. 
\end{proof}
We say that an $n\times n$ matrix  $X=(x_{ij})$ is \textit{persymmetric} if $x_{ij}=x_{n-j+1,n-i+1}$ for all $1\le i,j\le n$. Using the rotation operator $X \mapsto X^\mathsf{R}$ that sends a persymmetric matrix to a symmetric matrix and vice versa, we immediately deduce the following.
\begin{corollary}
Let $n\ge 2$. There exists $A\in \mathbb{C}^{n\times n}$ that cannot be expressed as a product of $r$ persymmetric Hankel matrices for any $r\in \mathbb{N}$.
\end{corollary}

Lastly we show that the Toeplitz factors in a Toeplitz decomposition cannot be restricted to circulant matrices. Recall that a matrix $X$ is \textit{circulant} if it takes the form
\[
X=\begin{bmatrix}
x_0 & x_{n-1} & \cdots & x_2 & x_1 \\
x_1 & x_0       & \cdots & x_3 & x_2\\
\vdots & \vdots &\ddots & \vdots & \vdots\\
x_{n-2} & x_{n-3}&\cdots & x_{0} & x_{n-1}\\
x_{n-1} & x_{n-2}& \cdots & x_1 & x_0
\end{bmatrix}.
\]
The set of all circulant matrices is an $n$-dimensional linear subspace of $\operatorname{Toep}_n(\mathbb{C})$. 
\begin{proposition}
Let $n\ge 2$. There exists $A\in \mathbb{C}^{n\times n}$ that cannot be expressed as a product of $r$ circulant matrices for any $r\in \mathbb{N}$.
\end{proposition}
\begin{proof}
Let $X$ be a circulant matrix. Then $v=[1,1,\dots, 1]^\mathsf{T}$ is an eigenvector of $X$. If every $n\times n$ matrix is a product of $r$ circulant matrices for some $r\in \mathbb{N}$, then $v$ must be an eigenvector of every $n \times n$ matrix, which is evidently false.
\end{proof}

\section{Computing the Toeplitz and Hankel decompositions}

We will discuss two approaches toward computing Toeplitz decompositions for generic matrices. The first uses numerical algebraic geometry and yields a decomposition with the minimal number, i.e., $r = \lfloor n/2 \rfloor +1$, of factors but is difficult to compute in practice. The second uses numerical linear algebra and is $O(n^3)$ in time complexity but requires an additional $n$ permutation matrices and yields a decomposition with $2n$ Toeplitz factors.

These proposed methods are intended to: (i) provide an idea of how  the purely existential discussions in  Sections~\ref{sec:toep} and \ref{sec:hank} may be made constructive, and (ii) shed light on the computational complexity of such decompositions (e.g.\ the second method is clearly polynomial time). Important issues like backward numerical stability have been omitted from our considerations. Further developments are necessary before these methods can become practical for large $n$ and these will be explored in \cite{YL}.

\subsection{Solving a system of linear and quadratic equations}

For notational convenience later, we drop the subscript $r$ and write 
\[
\rho:\operatorname{Toep}_n^r(\mathbb{C})\to \mathbb{C}^{n\times n}
\]
for the map $\rho_r$ we introduced in Theorem~\ref{thm:main}.
We observe that $\operatorname{Toep}_n(\mathbb{C}) \simeq \mathbb{C}^{2n-1}$ and therefore we may embed $\operatorname{Toep}_n^r(\mathbb{C})$, being a product of $r$ copies
of $\operatorname{Toep}_n(\mathbb{C})$, via the \textit{Segre embedding} \cite{Land} into $(\mathbb{C}^{2n-1})^{\otimes r}$. It is easy to see that we then have the following factorization of $\rho$:
\[
\begin{tikzcd}
\operatorname{Toep}_{n}(\mathbb{C})\times\dots\times \operatorname{Toep}_{n}(\mathbb{C}) \arrow{r}{\rho} \arrow[swap]{d}{i} & \mathbb{C}^{n\times n} \\
\mathbb{C}^{2n-1}\otimes\dots\otimes \mathbb{C}^{2n-1} \arrow[swap]{ur}{\pi}
\end{tikzcd}
\]
Here $i$ denotes the Segre embedding of $\operatorname{Toep}_n^r(\mathbb{C})$ into ${(\mathbb{C}^{2n-1})}^{\otimes r}$ and $\pi$ is a linear projection from ${(\mathbb{C}^{2n-1})}^{\otimes r}$ onto $\mathbb{C}^{n\times n}$. The image of the Segre embedding is the well-known \textit{Segre variety}. Note that like $\rho$, both $i$ and $\pi$ depend on $r$ but we omitted subscripts to avoid notational clutter. An explicit expression for $i$ is as an outer product $i (t_1,\dots,t_r) = t_1 \otimes \dots \otimes t_r$ where $t_1,\dots, t_r \in \mathbb{C}^{2n-1}$ are the vectors of parameters (e.g.\ first column and row) that determine the Toeplitz matrices $T_1,\dots, T_r$ respectively.  There is no general expression for $\pi$, but for a fixed $r$ one can determine $\pi$ iteratively.

For example, if $n=2$,  we set $r=2$ so that $\rho$ is dominant by  Theorem~\ref{thm:main}. Let $X,Y$ be two Toeplitz matrices, then
\[
X=\begin{bmatrix}
x_0 & x_1\\
x_{-1}& x_0
\end{bmatrix}, \quad Y=\begin{bmatrix}
y_0 & y_1\\
y_{-1}& y_0
\end{bmatrix}, \quad
X Y=\begin{bmatrix}
x_0y_0+x_1y_{-1} & x_0y_1+x_1y_0\\
x_{-1}y_0+x_0y_{-1} & x_{-1}y_1+x_0y_0
\end{bmatrix}.
\]
The map $\rho: \operatorname{Toep}_2(\mathbb{C})\times \operatorname{Toep}_2(\mathbb{C})\to \mathbb{C}^{2\times 2}$ can be factored as $\rho=\pi\circ i$,\[
\begin{tikzcd}
\operatorname{Toep}_2(\mathbb{C})\times \operatorname{Toep}_2(\mathbb{C})\arrow{r}{\rho} \arrow[swap]{d}{i} & \mathbb{C}^{2\times 2} \\
\mathbb{C}^{3}\otimes \mathbb{C}^3 \arrow[swap]{ur}{\pi}
\end{tikzcd}
\]
where $i$ is the Segre embedding of $\operatorname{Toep}_2(\mathbb{C})\times \operatorname{Toep}_2(\mathbb{C})$ into $\mathbb{C}^3\otimes \mathbb{C}^3$ and $\pi$ is the projection of $\mathbb{C}^3\otimes \mathbb{C}^3$ onto $\mathbb{C}^{2\times 2}$. More specifically we have
\[
i\left(
\begin{bmatrix}
x_0 & x_1\\
x_{-1}& x_0
\end{bmatrix},
\begin{bmatrix}
y_0 & y_1\\
y_{-1}& y_0
\end{bmatrix}\right)
=
\begin{bmatrix}
x_{-1}y_{-1} &  x_{-1}y_0 &  x_{-1}y_1\\
x_{0}y_{-1} & x_0y_0 & x_0y_1\\
x_1y_{-1} & x_1y_0 &x_1y_{1}
\end{bmatrix}
\]
and 
\[
\pi\left(
\begin{bmatrix}
z_{-1,-1} &  z_{-1,0} & z_{-1,1}\\
z_{0,-1} & z_{0,0} & z_{0,1}\\
z_{1,-1} & z_{1,0} & z_{1,1}
\end{bmatrix}\right)
=\begin{bmatrix}
z_{0,0}+z_{1,-1} & z_{0,1}+z_{1,0}\\
z_{-1,0}+z_{0,-1} & z_{-1,1}+z_{0,0}
\end{bmatrix}.
\]
Now given a $2\times 2$ matrix $A$, a decomposition of $A$ into the product of two Toeplitz matrices is equivalent to finding an intersection of the Segre variety $V=i(\operatorname{Toep}_2(\mathbb{C})\times \operatorname{Toep}_2(\mathbb{C}))$ with the affine linear space $\pi^{-1}(A)$. It is well-known that the Segre variety $V$ is cut out by quadratic equations given  by the vanishing of $2\times 2$ minors of
\[
\begin{bmatrix}
z_{-1,-1} & z_{-1,0} & z_{-1,1}\\
z_{0,-1}  & z_{0,0} & z_{0,1}\\
z_{1,-1} &z_{1,0} &z_{1,1}
\end{bmatrix}.
\]
These nine quadratic equations defined by the vanishing of $2 \times 2$ minors, together with the four linear equations that define $\pi^{-1}(A)$ can be used to calculate the decomposition of $A$. In summary, the problem of computing a Toeplitz decomposition of a $2\times 2$ matrix is reduced the problem of computing a solution to a system of nine quadratic and four linear equations. More generally this extends to arbitrary dimensions --- computing a Toeplitz decomposition of an $n \times n$ matrix is equivalent to computing a solution to  a linear-quadratic system
\begin{equation}\label{lqsys}
c_i^\mathsf{T} x = d_i, \quad i = 1,\dots,l,\qquad
x^\mathsf{T} E_j x = 0, \quad  j = 1,\dots,q.
\end{equation}
The $l$ linear equations form a linear system $C^\mathsf{T} x = d$ where $c_1,\dots,c_l$ are the columns of the matrix $C$ and $d = [d_1,\dots,d_l]^\mathsf{T}$; these define the linear variety $\pi^{-1}(A)$. The $q$ quadratic equations define the Segre variety $V$. By Theorem~\ref{thm:main}, the two varieties must have a nonempty intersection, i.e., a solution to \eqref{lqsys} must necessarily exist,  for any generic $A$ (and for all $A$ if Conjecture~\ref{conj} is true).
Observe that $d$ depends on the entries of the input matrix $A$ but the matrix $C$ and the symmetric matrices $E_1,\dots,E_q$ depend only on $r$ and are the same regardless of the input matrix $A$.

Such a system may be solved symbolically using computer algebra techniques (e.g.\ Macaulay2 \cite{macaulay}) or numerically via homotopy continuation techniques (e.g.\ Bertini \cite{bertini}). The complexity of solving  \eqref{lqsys}  evidently depends on both $l$ and $q$ but is dominated by $q$, the number of quadratic equations. It turns out that $q$ may often be reduced, i.e., some of the quadratic equations may be dropped from \eqref{lqsys}. For example, suppose  the entries of $X$ and $Y$ are all nonzero in the $2 \times 2$ example above. Observe that the linear equations defining $\pi^{-1}(A)$ do not involve  $z_{-1,-1}$ and $z_{1,1}$. So instead of the original system of \textit{nine} quadratic and four linear equations, we just need to consider a reduced system of \textit{two} quadratic equations $z_{-1,0}z_{0,1}-z_{0,0}z_{-1,1}=0$, $z_{0,-1}z_{1,0}-z_{0,0}z_{1,-1}=0$ and four linear equations.

In the $3 \times 3$ case, $\rho$ factorizes as
\[
\begin{tikzcd}
\operatorname{Toep}_3(\mathbb{C})\times \operatorname{Toep}_3(\mathbb{C})\arrow{r}{\rho} \arrow[swap]{d}{i} & \mathbb{C}^{3\times 3} \\
\mathbb{C}^{5}\otimes \mathbb{C}^5 \arrow[swap]{ur}{\pi}
\end{tikzcd}
\]
Denoting the Toeplitz factors by $X=[x_{j-i}], Y=[y_{j-i}] \in  \operatorname{Toep}_3(\mathbb{C}) $, the maps $i$ and $\pi$ are given by
\begin{equation}\label{eq:3x3q}
i([x_{k}],[y_m])=[x_{k}y_{m}]\in\mathbb{C}^{5 \times 5}, \quad k,m=-2,-1,0,1,2,
\end{equation}
and
\begin{equation}\label{eq:3x3l}
\pi ([z_{km}])=\left[\sum\nolimits_{k+m =j-i, \; 1-i\le k, \; m\le j-1} z_{km} \right] \in\mathbb{C}^{3\times 3}, \quad i,j = -1,0,1.
\end{equation}
The vanishing of the $2\times 2$ minors of \eqref{eq:3x3q} yields a system of ten quadratic equations and setting $\pi(Z) = A$  in \eqref{eq:3x3l} yields a system of nine linear equations. Any common solution, which must exist by Theorem~\ref{thm:main}, gives us a decomposition of the generic $3\times 3$ matrix $A$.

\begin{example}
The following is an explicit numerical example computed by solving the linear-quadratic system with Bertini \cite{bertini}.
\[
\begin{bmatrix}
1 &2& 3\\
4 &5& 6\\
7 &8& 9
\end{bmatrix}
=
 \begin{bmatrix}
2.2222&    0.8889 &  -0.4444\\
    3.5556    &2.2222 &   0.8889\\
    4.8889   & 3.5556  &  2.2222
\end{bmatrix}
\begin{bmatrix}
      0.2500 &   1.0000   & 1.0000\\
    1.0000    &0.2500   & 1.0000\\
    1.0000    &1.0000 &   0.2500
\end{bmatrix} .
 \]
\end{example}

While our discussion is about Toeplitz decomposition, the method described in this section applies verbatim to Hankel decomposition by factorizing $\rho = \rho_r$ in \eqref{eq:hdecom} instead.

\subsection{Using Gaussian elimination}

We will exhibit an algorithm to decompose a generic $n\times n$ matrix $A$ into a product of $2n$ Toeplitz matrices $T_1,\dots,T_{2n}$ and $n$ permutation matrices $P_1,\dots, P_n$ in the form
\begin{equation}\label{eq:ge}
A = T_1 T_2 P_1 T_3 T_4 P_2 \cdots  T_{2n-1} T_{2n} P_n.
\end{equation}

For simplicity, we assume that all computations are in exact arithmetic. In particular, note for a generic matrix we may perform Gaussian elimination without pivoting if we disregard numerical stability issues.
\begin{description}
\item[Input] Let $A=(a_{ij}) \in \mathbb{C}^{n \times n}$  be generic.

\item[Step 1] Using Gaussian elimination,  determine column vectors $v_1,v_2,\dots, v_n$ such that 
\[
A=(I+v_1 e_1^\mathsf{T})(I+v_2 e_2^\mathsf{T})\cdots (I+v_n e_n^\mathsf{T}),
\] 
where $e_j$ is the standard basis of $\mathbb{C}^n$ and $I$ is the $n\times n$ identity  matrix.
\item[Step 2] For each factor $I+v_k e_k^\mathsf{T}$, we write
\[
I+v_k e_k^\mathsf{T}=\Pi_k  (I+w_k e_1^\mathsf{T}) \Pi_k,
\]
where $\Pi_k$ is the permutation matrix corresponding to the permutation $(1\shortrightarrow k) \in \mathfrak{S}_n$ and $w_k=\Pi_k v_k$.
\item[Step 3] For each factor  $I+w_k e_1^\mathsf{T}$, where $w_k=[w_{k1}, w_{k2},\dots, w_{kn}]^\mathsf{T}$ we define
\[
W_k =\begin{bmatrix}
w_{kn}&w_{k,n-1} & w_{k,n-2} & \cdots &w_{k3} & w_{k2} & w_{k1}\\
0&w_{kn} & w_{k,n-1} & \cdots &w_{k4} & w_{k3}& w_{k2}\\
0&0 & w_{kn} & \cdots &w_{k5} & w_{k4}& w_{k3}\\
\vdots&\vdots& \vdots &\ddots& \vdots &\vdots &\vdots\\
0&0 & 0 &   & w_{kn} & w_{k,n-1}& w_{k,n-2}\\
0&0 & 0 & \cdots & 0 & w_{kn}& w_{k,n-1}\\
0&0 & 0 & \cdots & 0 & 0& w_{kn}\\
\end{bmatrix}.
\]
Then 
\[
I+w_k e_1^\mathsf{T}=W_k (W_k ^{-1}+E_{n1}) \quad
\text{where} \quad 
E_{n1} =
\begin{bmatrix}
0 & 0 & \cdots & 0 &0\\
0 & 0 & \cdots & 0 & 0\\
\vdots&\vdots &\ddots &\vdots &\vdots\\
0 & 0 & \cdots & 0 &0\\
1 & 0 & \cdots & 0 & 0
\end{bmatrix}.
\]
Take $T_k=W_k$ and $T_{k+1}=W_k^{-1}+E_{n1}$.
\item[Output] The required Toeplitz factors are $T_k$'s and $T_{k+1}$'s and the permutation factors are $P_k :=\Pi_k\Pi_{k+1}$ for $k=1,\dots, n$.
\end{description}
It is easy to see that this indeed gives a decomposition of $A$ into a product of Toeplitz matrices and permutation matrices as in \eqref{eq:ge}. Computational cost is dominated by the $O(n^3)$ arithmetic steps required in Gaussian elimination. Since the inversion of an upper-triangular Toeplitz matrix requires only $O(n \log n)$ arithmetic steps (cf. Section~\ref{sec:why}), this algorithm has $O(n^3)$ complexity.
\begin{example}
The following is an explicit numerical example computed using the algorithm described above. We generate a random $5\times 5$ matrix with small integer entries
\[
A=
\begin{bmatrix}
2 &5& 2& 5& 3\\
4 &5& 5& 2& 2\\
2 &3& 2& 1& 5\\
3 &1& 5& 2& 3\\
4 &1& 2& 4& 3\\
\end{bmatrix}.
\]
The output of the algorithm is reproduced below. We use the notation $T(v,w)$ to denote the Toeplitz matrix whose first row is $v \in \mathbb{C}^n$ and first column is $w \in \mathbb{C}^n$. We represent permutation matrices as elements in the symmetric group $\mathfrak{S}_n$.
{\footnotesize
\begin{align*}
T_1&=T([4,3,2,4,1],[4,0,0,0,0]),\\
T_2&=T([0.25,-0.1875,0.015625,-0.16796875,0.2431640625],[0.25,0,0,0,1]),\\
P_1&=(1 \shortrightarrow 2),\\
T_3&=T([-9,-6.5,-2,2.5,-6],[-9,0,0,0,0]),\\
T_4&=T([-0.11111,0.0802469,0.0332647,-0.02467230,0.121575],[-0.11111,0,0,0,1]),\\
P_2&=(1 \shortrightarrow 3 \shortrightarrow 2),\\
T_5&=T([-3.8,0.7,1.5,-0.2,-1.4],[-3.8,0,0,0,0]),\\
T_6&=T([-0.26316,-0.0484764, -0.112808,-0.026065, 0.050173],[-0.26315,0,0,0,1]),\\
P_3&=(1\shortrightarrow 4 \shortrightarrow 3),\\
T_7&=T([16,-4.5,2,2,2.5],[16,0,0,0,0]),\\
T_8&=T([-0.0625,0.017578,-0.0028687,-0.0108166,-0.014646],[0.0625,0,0,0,1]),\\
P_4&=(1 \shortrightarrow 5 \shortrightarrow 4),\\
T_9&=T([25.85714,2.85714,-14.71429,-6.7142857,-76.71429],[25.85714,0,0,0,0]),\\
T_{10}&=T([0.038674,-0.004273,0.02248,0.00513,0.125856],[0.038674,0,0,0,1]),\\
P_5&=(1 \shortrightarrow 5).
\end{align*}}%
We have
\[
A = T_1 T_2 P_1 T_3 T_4 P_2 T_5 T_6 P_3 T_7 T_8 P_4 T_9 T_{10} P_5.
\]
\end{example}

It is straightforward to modify the algorithm to obtain a decomposition of a generic matrix into a product of $2n$ Hankel matrices and $n+1$ permutation matrices. Let $\Pi = [\pi_{ij}]$ be the permutation matrix defined by
\[
\pi_{ij}=\begin{cases}
0& i+j\ne n+1,\\
1&i+j=n+1.
\end{cases}
\]
It is easy to see that multiplying a Toeplitz matrix $T_i$ on  either the left or right by $\Pi$ gives a Hankel matrix $H_i$ and vice versa. Now we may write the Toeplitz factors $T_i$'s in \eqref{eq:ge} as $T_i = \Pi H_i$ if $i$ is odd and  $T_i = H_i\Pi$ if $i$ is even. We also write $P'_i =\Pi P_i \Pi$ for $i=1,\dots,n-1$, and $P'_n = \Pi P_n$. Then we obtain
\[
A = T_1 T_2 P_1 T_3 T_4 P_2 \cdots  T_{2n-1} T_{2n} P_n
= \Pi H_1 H_2 P'_1 H_3 H_4 P'_2 \cdots H_{2n-1} H_{2n} P'_n
\]
as required.

\section*{Acknowledgments}
We thank Professor T.~Y.~Lam for inspiring this work. This article is dedicated to him on his 70th birthday. 
This work is partially supported by AFOSR FA9550-13-1-0133, NSF DMS 1209136, and NSF DMS 1057064.


\begin{thebibliography}{99}%
\bibitem {10}``Algorithms for the ages,'' \textit{Science}, \textbf{287}
(2000), no.~5454, pp.~799.

\bibitem{fast2} G.~S.~Ammar and W.~B.~Gragg, ``Superfast solution of real positive definite Toeplitz systems,'' \textit{SIAM J.\ Matrix Anal.\ Appl.}, \textbf{9} (1988), no.~1, pp.~61--76.

\bibitem{BS}Y.~Barnea and A.~Shalev, ``Hausdorff dimension, pro-{$p$} groups, and {K}ac-{M}oody algebras,'' \textit{Trans.\ Amer.\ Math.\ Soc.}, \textbf{349} (1997), no.~12, pp.~5073--5091.

\bibitem{bertini} D.~J.~Bates, J.~D.~Hauenstein, A.~J.~Sommese, and C.~W.~Wampler, \textit{Numerically Solving Polynomial Systems with Bertini}, Software, Environments, and Tools, \textbf{25}, SIAM, Philadelphia, PA, 2013.

\bibitem{inteqn}H.~Bart, I.~Gohberg, and M.~A.~Kaashoek, ``Wiener-Hopf integral equations, Toeplitz matrices and linear systems,'' pp.~85--135, \textit{Operator Theory: Adv.\ Appl.}, \textbf{4}, Birkh\"{a}user, Boston, MA, 1982.

\bibitem{alg}J.~Bernik, R.~Drnov\v{s}ek, D.~Kokol Bukov\v{s}ek, T.~Ko\v{s}ir, M.~Omladi\v{c}, and H.~Radjavi, ``On semitransitive Jordan algebras of matrices,'' \textit{J.\ Algebra Appl.}, \textbf{10} (2011), no.~2, pp.~319--333.

\bibitem{queue}D.~Bini and B.~Meini, ``Solving certain queueing problems modelled by Toeplitz matrices,'' \textit{Calcolo}, \textbf{30} (1993), no.~4, pp.~395--420.

\bibitem{fast0}R.~B.~Bitmead and B.~D.~O.~Anderson, ``Asymptotically fast solution of Toeplitz and related systems of linear equations,'' 
\textit{Linear Algebra Appl.}, \textbf{34} (1980), pp.~103--116. 

\bibitem{quantize}M.~Bordemann, E.~Meinrenken, and M.~Schlichenmaier, ``Toeplitz quantization of K\"{a}hler manifolds and $\mathfrak{gl}(N)$, $N\to\infty$ limits,'' \textit{Comm.\ Math.\ Phys.}, \textbf{165} (1994), no.~2, pp.~281--296. 

\bibitem {Borel}A.~Borel, \textit{Linear Algebraic Groups}, 2nd Ed., Graduate Texts in Mathematics, \textbf{126}, Springer-Verlag, New York, NY, 1991.

\bibitem{toepdet1}A.~Borodin and A.~Okounkov, ``A Fredholm determinant formula for Toeplitz determinants,''
\textit{Integral Equations Operator Theory}, \textbf{37} (2000), no.~4, pp.~386--396. 

\bibitem{toepop} D.~Burns and V.~Guillemin, ``The Tian--Yau--Zelditch theorem and Toeplitz operators,'' \textit{J.\ Inst.\ Math.\ Jussieu}, \textbf{10} (2011), no.~3, pp.~449--461.

\bibitem{iter1}R.~H.~Chan and X.-Q.~Jin, \textit{An Introduction to Iterative Toeplitz Solvers}, Fundamentals of Algorithms, \textbf{5}, SIAM, Philadelphia, PA, 2007.

\bibitem{iter0}R.~H.~Chan and M.~K.~Ng, ``Conjugate gradient methods for Toeplitz systems,'' \textit{SIAM Rev.}, \textbf{38} (1996), no.~3, pp.~427--482.

\bibitem{cond1}R.~H.~Chan and G.~Strang, ``Toeplitz equations by conjugate gradients with circulant preconditioner,'' \textit{SIAM J.\ Sci.\ Statist.\ Comput.}, \textbf{10} (1989), no.~1, pp.~104--119. 

\bibitem{cond2}T.~F.~Chan, ``An optimal circulant preconditioner for Toeplitz systems,'' \textit{SIAM J.\ Sci\. Statist.\ Comput.}, \textbf{9} (1988), no.~4, pp.~766--771.

\bibitem {fast5}S.~Chandrasekaran, M.~Gu, X.~Sun, J.~Xia, and J.~Zhu, ``A superfast algorithm for Toeplitz systems of linear equations,'' \textit{SIAM J.\ Matrix Anal.\ Appl.}, \textbf{29} (2007), no.~4, pp.~1247--1266.

\bibitem {time}W.~W.~Chen, C.~M.~Hurvich, and Y.~Lu, ``On the correlation matrix of the discrete Fourier transform and the fast solution of large Toeplitz systems for long-memory time series,'' \textit{J.\ Amer.\ Statist.\ Assoc.}, \textbf{101} (2006), no.~474, pp.~812--822.

\bibitem{image} W.~K.~Cochran, R.~J.~Plemmons, and T.~C.~Torgersen, ``Exploiting Toeplitz structure in atmospheric image restoration,'' pp.~177--189, \textit{Structured Matrices in Mathematics, Computer Science, and Engineering I}, \textit{Contemp.\ Math.}, \textbf{280}, AMS, Providence, RI, 2001. 

\bibitem {fft}J.~W.~Cooley and J.~W.~Tukey, ``An algorithm for the machine calculation of complex Fourier series,'' \textit{Math.\ Comp.}, \textbf{19} (1965), no.~90, pp.~297--301.

\bibitem{fast1}F.~de Hoog, ``A new algorithm for solving Toeplitz systems of equations,'' \textit{Linear Algebra Appl.}, \textbf{88/89} (1987), pp.~123--138. 

\bibitem{toepdet2}P.~Deift, A.~Its, and I.~Krasovsky, ``Asymptotics of Toeplitz, Hankel, and Toeplitz+Hankel determinants with Fisher--Hartwig singularities,'' \textit{Ann.\ Math.}, \textbf{174} (2011), no.~2, pp.~1243--1299.

\bibitem{stat}A.~Dembo, C.~L.~Mallows, and L.~A.~Shepp, ``Embedding nonnegative definite Toeplitz matrices in nonnegative definite circulant matrices, with application to covariance estimation,'' \textit{IEEE Trans.\ Inform.\ Theory}, \textbf{35} (1989), no.~6, pp.~1206--1212.

\bibitem{opalg}R.~G.~Douglas and R.~Howe, ``On the $C^*$-algebra of Toeplitz operators on the quarterplane,'' \textit{Trans.\ Amer.\ Math.\ Soc.}, \textbf{158} (1971), pp.~203--217.

\bibitem{quantum} E.~Eisenberg, A.~Baram, and M.~Baer, ``Calculation of the density of states using discrete variable representation and Toeplitz matrices,''
\textit{J.\ Phys.\ A}, \textbf{28} (1995), no.~16, pp.~L433--L438. 

\bibitem{repth} M.~Engli\v{s}, ``Toeplitz operators and group representations,'' \textit{J.\ Fourier Anal.\ Appl.}, \textbf{13} (2007), no.~3, pp.~243--265. 

\bibitem{graph}R.~Euler, ``Characterizing bipartite Toeplitz graphs,'' \textit{Theoret.\ Comput.\ Sci.}, \textbf{263} (2001), no.~1--2, pp.~47--58.

\bibitem{fast6}P.~Favati, G.~Lotti, and O.~Menchi, ``A divide and conquer algorithm for the superfast solution of Toeplitz-like systems,'' \textit{SIAM J.\ Matrix Anal.\ Appl.}, \textbf{33} (2012), no.~4, pp.~1039--1056.

\bibitem{Geck}M.~Geck, \textit{An Introduction to Algebraic Geometry and Algebraic Groups}, Oxford Graduate Texts in Mathematics, \textbf{10}, Oxford University Press, Oxford, 2003.

\bibitem{qdeform} V.~Gorin, ``The $q$-Gelfand-Tsetlin graph, Gibbs measures and $q$-Toeplitz matrices,'' \textit{Adv.\ Math.}, \textbf{229} (2012), no.~1, pp.~201--266. 

\bibitem{macaulay} D.~R.~Grayson and M.~E.~Stillman, \textit{Macaulay2: A software system for research in algebraic geometry}, available at \url{http://www.math.uiuc.edu/Macaulay2/}, 2002.

\bibitem {fmm}L.~Greengard and V.~Rokhlin, ``A fast algorithm for particle simulations,'' \textit{J.\ Comput.\ Phys.}, \textbf{73} (1987), no.~2, pp.~325--348.

\bibitem{anlys}U.~Grenander and G.~Szeg\"{o}, \textit{Toeplitz Forms and Their Applications}, 2nd Ed., Chelsea Publishing, New York, NY, 1984.

\bibitem{compress}J.~Haupt, W.~U.~Bajwa, G.~Raz, and R.~Nowak, ``Toeplitz compressed sensing matrices with applications to sparse channel estimation,'' \textit{IEEE Trans.\ Inform.\ Theory}, \textbf{56} (2010), no.~11, pp.~5862--5875.

\bibitem {Howe}R.~Howe, ``Very basic Lie theory,'' \textit{Amer.\ Math.\ Monthly}, \textbf{90} (1983), no.~9, pp.~600--623.

\bibitem{fast7a} J.-J.~Hsue and A.~E.~Yagle, ``Fast algorithms for solving Toeplitz systems of equations using number-theoretic transforms,'' \textit{Signal Process.}, \textbf{44} (1995), no.~1, pp.~89--101.

\bibitem{comb}M.~Kac, ``Some combinatorial aspects of the theory of Toeplitz matrices,'' pp.~199--208, \textit{Proc.\ IBM Sci.\ Comput.\ Sympos.\ Combinatorial Problems}, IBM Data Process.\ Division, White Plains, NY, 1964.

\bibitem{fast7b}H.~Khalil, B.~Mourrain, and M.~Schatzman,  ``Superfast solution of Toeplitz systems based on syzygy reduction,'' \textit{Linear Algebra Appl.}, \textbf{438} (2013), no.~9, pp.~3563--3575. 

\bibitem {Knapp}A.~W.~Knapp, \textit{Lie Groups Beyond an Introduction}, 2nd Ed., Progress in Mathematics, \textbf{140}, Birkh\"{a}user, Boston, MA, 2002.

\bibitem {Land}J.~M.~Landsberg, \textit{Tensors: Geometry and Applications},
AMS, Providence, RI, 2012.

\bibitem{numie}F.-R.~Lin, M.~K.~Ng, and R.~H.~Chan, ``Preconditioners for Wiener-Hopf equations with high-order quadrature rules,'' \textit{SIAM J.\ Numer.\ Anal.}, \textbf{34} (1997), no.~4, pp.~1418--1431.

\bibitem{diffgeom}X.~Ma and G.~Marinescu, ``Toeplitz operators on symplectic manifolds,'' \textit{J.\ Geom.\ Anal.}, \textbf{18} (2008), no.~2, pp.~565--611. 

\bibitem{toepker}N.~Makarov and A.~Poltoratski, ``Beurling--Malliavin theory for Toeplitz kernels,'' \textit{Invent.\ Math.}, \textbf{180} (2010), no.~3, pp.~443--480. 

\bibitem{top}R.~J.~Milgram, ``The structure of spaces of Toeplitz matrices,'' \textit{Topology}, \textbf{36} (1997), no.~5, pp.~1155--1192.

\bibitem{iter2}M.~K.~Ng, \textit{Iterative Methods for Toeplitz Systems}, Oxford University Press, New York, NY, 2004.

\bibitem{control}H.~\"{O}zbay and A.~Tannenbaum, ``A skew Toeplitz approach to the $H^\infty$ optimal control of multivariable distributed systems,'' \textit{SIAM J.\ Control Optim.}, \textbf{28} (1990), no.~3, pp.~653--670. 

\bibitem{prob}D.~Poland, ``Toeplitz matrices and random walks with memory,'' \textit{Phys.\ A}, \textbf{223} (1996), no.~1--2, pp.~113--124.

\bibitem{alggeom}K.~Rietsch, ``Totally positive Toeplitz matrices and quantum cohomology of partial flag varieties,'' \textit{J.\ Amer.\ Math.\ Soc.}, \textbf{16} (2003), no.~2, pp.~363--392.

\bibitem {Decomp}G.~W.~Stewart, ``The decompositional approach to matrix computation,'' \textit{Comput.\ Sci.\ Eng.}, \textbf{2} (2000), no.~1, pp.~50--59.

\bibitem{fast4}M.~Stewart, ``A superfast Toeplitz solver with improved numerical stability,'' \textit{SIAM J.\ Matrix Anal.\ Appl.}, \textbf{25} (2003), no.~3, pp.~669--693.

\bibitem {Taylor}J.~L.~Taylor, \emph{Several Complex Variables with Connections to Algebraic Geometry and Lie groups}, Graduate Studies in Mathematics, \textbf{46}, AMS, Providence, RI, 2002.

\bibitem{pde} J.~Toft, ``The Bargmann transform on modulation and Gelfand--Shilov spaces, with applications to Toeplitz and pseudo-differential operators,'' \textit{J.\ Pseudo-Differ.\ Oper.\ Appl.}, \textbf{3} (2012), no.~2, pp.~145--227.

\bibitem{numint}E.~E.~Tyrtyshnikov, ``Fast computation of Toeplitz forms and some multidimensional integrals,'' \textit{Russian J.\ Numer.\ Anal.\ Math. Modelling}, \textbf{20} (2005), no.~4, pp.~383--390.

\bibitem{numpde} S.~Serra, ``The rate of convergence of Toeplitz based PCG methods for second order nonlinear boundary value problems,'' \textit{Numer.\ Math.}, \textbf{81} (1999), no.~3, pp.~461--495.

\bibitem{signal} U.~Steimel, ``Fast computation of Toeplitz forms under narrowband conditions with applications to statistical signal processing,'' \textit{Signal Process.}, \textbf{1} (1979), no.~2, pp.~141--158.

\bibitem{approx}G.~Strang, ``The discrete cosine transform, block Toeplitz matrices, and wavelets,'' pp.~517--536, \textit{Advances in Computational Mathematics},  Lecture Notes in Pure and Applied Mathematics, \textbf{202}, Dekker, New York, NY, 1999.

\bibitem{cond0}G.~Strang, ``A proposal for Toeplitz matrix calculations,'' \textit{Stud.\ Appl.\ Math.}, \textbf{74} (1986), no.~2, pp.~171--176.

\bibitem{fast3}M.~Van~Barel, G.~Heinig, and P.~Kravanja, ``A stabilized superfast solver for nonsymmetric Toeplitz systems,'' \textit{SIAM J.\ Matrix Anal.\ Appl.}, \textbf{23} (2001), no.~2, pp.~494--510.

\bibitem{YL}K.~Ye and L.-H.~Lim, ``New classes of matrix decompositions,'' \textit{preprint}, (2014).
\end{thebibliography}
\end{document}